\documentclass[11pt]{amsart}
\usepackage{lmodern}

\usepackage[utf8]{inputenc}
\usepackage[a4paper]{geometry}
\usepackage{units}
\usepackage{amsthm}
\usepackage{amssymb}
\usepackage{esint}
\usepackage{hyperref}
\hypersetup{
 pdftex,colorlinks}

\makeatletter
\numberwithin{equation}{section}
\numberwithin{figure}{section}

\renewcommand{\[}{\begin{equation}}
\renewcommand{\]}{\end{equation}} 
\usepackage{mathrsfs}
\usepackage{latexsym,lmodern,graphicx,enumitem,amssymb,dsfont,units}

\numberwithin{equation}{section}
\theoremstyle{plain}
\newtheorem{thm}{Theorem}[section]
  \theoremstyle{remark}
  \newtheorem{rem}[thm]{Remark}
  \theoremstyle{plain}
  \newtheorem*{thm*}{Theorem}
  \theoremstyle{plain}
  \newtheorem{prop}[thm]{Proposition}
  \theoremstyle{plain}
  \newtheorem{lem}[thm]{Lemma}
  \theoremstyle{plain}
  \newtheorem{cor}[thm]{Corollary}

\makeatother

\begin{document}

\title{Diffusive decay of the environment viewed by the particle}

\date{\today}

\author[P. de Buyer]{Paul de Buyer}
\address[P. de Buyer]{Université Paris Ouest Nanterre La Défense - Modal'X, 200 avenue
de la République 92000 Nanterre, France}
\email{debuyer@math.cnrs.fr}

\author[J.-C. Mourrat]{Jean-Christophe Mourrat}
\address[J.-C. Mourrat]{ENS Lyon, CNRS, 46 all\'ee d'Italie, 69007 Lyon, France}
\email{jean-christophe.mourrat@ens-lyon.fr}



\maketitle

\begin{abstract}
We prove an optimal diffusive decay of the environment viewed by the
particle in random walk among random independent conductances, with,
as a main assumption, finite second moment of the conductance. Our
proof, using the analytic approach of Gloria, Neukamm and Otto, is
very short and elementary. 
\end{abstract}

\section{Introduction}

In this paper we are interested in the speed of convergence of the
environment viewed by the particle in the context of random walk among
random independent conductances. We refer to \cite{biskup2011recent}
for an introduction. We will first introduce the model and then state
our main theorem.

Consider the $d$-dimensional lattice $(\mathbb{Z}^{d},\mathbb{B}^{d})$,
$d\geq1$, with $\mathbb{B}^{d}$ the set of unoriented edges connecting
any two points of $\mathbb{Z}^{d}$ at Euclidian distance 1 and set
$\Omega:=[0,+\infty)^{\mathbb{B}^{d}}$. Given an \emph{environment}
$\omega=\left(\omega_{x,y}\right)_{\{x,y\}\in\mathbb{B}^{d}}\in\Omega$
we shall consider the associated Markov process $\left(X_{t}^{\omega}\right)_{t\geqslant0}$
with jump rate between $x$ and $y$ given by the \emph{conductance}
$\omega_{x,y}$ and write $P_{x}^{\omega}$ for its law starting from
$x\in\mathbb{Z}^{d}$. In many places we may simply write $\left(X_{t}\right)_{t\geqslant0}$
for simplicity. 
The environment itself will also be a random variable. In fact, throughout
the paper, we will assume that the conductances $\omega_{x,y}$, $\{x,y\}\in\mathbb{B}^{d}$,
are i.i.d.\ with common law $\mu$, whose support is included in
$[1,\infty)$. The law of the environment is therefore the product
measure $\mathbb{P}:=\mu^{\mathbb{B}^{d}}$ and we denote by $\mathbb{E}$
the associated expectation. 
Since $\mathbb{P}$ is a product measure, standard percolation arguments
guarantee that that the Markov process $\left(X_{t}\right)_{t\geqslant0}$
is well defined for almost all $\omega\in\Omega$, for all times,
see \textit{e.g.}\ \cite{kipnis1986central, DFGW, Mo99}.
Moreover it is reversible with respect to the counting measure, \textit{i.e.}\ for
all $x,y\in\mathbb{Z}^{d}$, it holds $P_{x}^{\omega}\left(X_{t}=y\right)=P_{y}^{\omega}\left(X_{t}=x\right)$.

Now, define the translation operators $\left(\theta_{z}\right)_{z\in\mathbb{Z}^{d}}$
given by $\left(\theta_{z}\omega\right)_{x,y}:=\omega_{x+z,y+z}$.
Then the Markov process $\left(\omega\left(t\right)\right)_{t\geqslant0}:=\left(\theta_{X_{t}}\omega\right)_{t\geqslant0}$,
called \emph{the environment viewed by the particle}, is reversible
with respect to $\mathbb{P}$.


The aim of this short paper is to give an optimal quantitative bound
on the decay to equilibrium of $\left(\omega\left(t\right)\right)_{t\geqslant0}$.

In order to state our theorem, we need some more notations. A function
$f\colon\mathbb{Z}^{d}\times\Omega\rightarrow\mathbb{R}$ is said
to be \emph{translation invariant} if $f\left(x,\omega\right)=f\left(0,\theta_{x}\omega\right)$
for all $x\in\mathbb{Z}^{d}$ and \emph{local} if $f(0,\cdot)$ depends
only on a finite set of conductances. The smallest set satisfying
that property is called the support of $f$ and is denoted by $supp(f)$,
while $\#supp\left(f\right)$ stands for the size of the support (\textit{i.e.}\ the
number of sites of $\mathbb{Z}^{d}$ contained in $supp(f)$). For
any translation invariant function $f$, we set $\mathbb{E}[f]:=\mathbb{E}[f(0,\cdot)]$.

Our main theorem is the following.

\begin{thm} \label{thm:MainTheorem} Let $d\geq3$. Assume that the
law $\mu$ of the conductance has support in $[1,\infty)$ and finite
second moment $\mathbb{E}\left[\omega_{\cdot,\cdot}^{2}\right]<\infty$.
Then there exists a constant $C>0$ that depends only on $d$ such
that for all local translation invariant function $f\colon\mathbb{Z}^{d}\times\Omega\rightarrow\mathbb{R}$
with $\mathbb{E}\left[f\right]=0$, all $x\in\mathbb{Z}^{d}$, it
holds 
\begin{equation}
\mathbb{E}\left(E_{x}^{\omega}\left[f\left(X_{t},\omega\right)\right]^{2}\right)\leqslant C\mathbb{E}\left[\omega_{\cdot,\cdot}^{2}\right]\#supp\left(f\right)^{2}\frac{\mathbb{E}\left[f^{2}\right]}{t^{d/2}}\qquad\forall t>0.\label{eq:main}
\end{equation}

\end{thm}

Let us comment on Theorem \ref{thm:MainTheorem}.

We first observe that similar results already exist in the literature.
One of us \cite{Mo99} proved a polynomial decay in dimension $1$ with
exponent $1/2$, in dimension 2 with behavior $\log t/t$, in dimension
3 to 6 with $t^{-1}$ and in dimension 7 and higher with $t^{\frac{-d}{2}+2}$.
Moreover, in the right hand side of \eqref{eq:main} appears a much
stronger norm than only the $L^{2}$ norm (the sum of the $L^{\infty}$
norm of $f$ and the so-called triple norm $|||f|||$ that involves
the infinite sum of $L^{\infty}$-norm of local gradients of $f$,
a natural norm largely used in the statistical mechanics literature,
see for example \cite{martinelli} and also \cite{bertini,janvresse,roberto}).
On the other hand, the results of \cite{Mo99} hold without the assumption on
the finiteness of the second moment of the conductance. More recently,
Gloria, Neukamm and Otto \cite{gloria} obtained, among other results,
a polynomial decay $t^{-\frac{d}{2}-1}$ when the function $f$ is
the divergence of some other function. Because we follow closely their
approach, our result is not far from theirs, although they need to
assume that conductances are bounded, \textit{i.e.}\ that $\mu$
has support contained in a compact set and that the function $f$
considered is in $L^{p}$, where $p$ is very large and not explicit. Finally,
we mention that \cite{yurinskii,cudna,gloria-otto11,gloria-otto12}
are related papers that deal with estimates on the diffusion matrix
and we refer to \cite{new-gno,am} and references therein for recent
results on homogenization of random operators.

\smallskip{}

Then, we observe that the power in \eqref{eq:main} is the best possible,
and we note that it can be obtained under stronger assumptions in \cite{gloria}.
On the other hand, we believe that the result should hold without
the assumption on the finiteness of the second moment of the conductance,
which appears as a technical fact in our proof. The second moment
$\mathbb{E}\left[f^{2}\right]$ is also best possible. Indeed, a smaller
norm would imply some regularization effect that would in turn imply
a spectral gap estimate which is known not to hold. For the same reason,
there must be some dependence, in the right hand side of \eqref{eq:main},
on the size of the support of $f$. Finally, we observe that the assumption
on the support of $\mu$ is also necessary since it is known, for
some specific choice of function $f$, with conductances taking small
values, that the decay of the heat-kernel coefficients can be very
slow, see \cite{berger2008anomalous}, which would interfer with the
decay in \eqref{eq:main}.

\smallskip{}

As for the proof, diffusive scaling is usually obtained, including
other settings on graphs such as interacting particle systems (Kawasaki
dynamics), using functional inequalities of Nash type,
see \textit{e.g.}\ \cite{liggett,bertini,Mo99}, Harnack \cite{delmotte}
or weak Poincaré inequalities \cite{bobkov,roberto}, see also \cite{bakry,saloff}.
Induction techniques can also be used \cite{janvresse,cesi}, or specific
aspects of the model such as attractivness \cite{deuschel}. However,
none of these techniques seem, to the best of our knowledge, to apply
to our setting. Here instead, we will mainly follow the PDEs ideas
from \cite{gloria}, but with many simplifications due to our non
specific choice of class of functions.

\smallskip{}

As a result we believe that our approach of the decay to equilibrium
of the environment viewed by the particle improves upon known results
into different directions (class of function, assumption
on the conductance, short and elementary proof), but, as a counterpart,
we are not able to extend the result of Gloria, Neukamm and Otto \cite{gloria}
to unbounded conductances for divergence functions. 

\medskip{}

The paper is organized as follows. In the next subsection we will
give some more notations. Then, we will give the proof of Theorem
\ref{thm:MainTheorem}. Such a proof will rely on a series of lemmas
that we will prove later on. Finally, in the last section, using the
theory of completely monotonic functions, we comment on a possible
short way of extending the results of Gloria, Neukamm and Otto \cite{gloria}
to our setting.

\subsection{Notations}

In this section we give some more notations.

Given an environment $\omega\in\Omega$, the infinitesimal generator
of the process $(X_{t})_{t\geq0}$ and its associated semi-group,
acting on functions $f\colon\mathbb{Z}^{d}\times\Omega\to\mathbb{R}$
are given, for any $x,\omega\in\mathbb{Z}^{d}\times\Omega$ by 
\[
\mathcal{L}^{\omega}f\left(x,\omega\right)=\sum_{\left|z\right|=1}\omega_{x,x+z}\left(f\left(x+z,\omega\right)-f\left(x,\omega\right)\right), \notag
\]
respectively by 
\[
P_{t}^{\omega}f\left(x,\omega\right)=\exp\left(\mathcal{L}^{\omega}t\right)f\left(x,\omega\right)=\sum_{y\in\mathbb{Z}^{d}}P_{x}^{\omega}\left(X_{t}=y\right)f\left(y,\omega\right). \notag
\]

In some situations, we may want to work with a given fixed environment.
In that case, and to emphasize this fact, we shall use the letter
$m$ instead of $\omega$ and call that given environment \emph{the
walk scheme}. A walk scheme is well-defined when there is a threshold 
such that the set of conductances above it does not percolate. 
This will be used in particular to evaluate the behaviour
of $\left(X_{t}^{m}\right)_{t\geqslant0}$.

In many places we shall use the following equality $P_{0}^{\theta_{x}\omega}\left(X_{t}=y\right)=P_{x}^{\omega}\left(X_{t}=y+x\right)$
that holds for all $x,y\in\mathbb{Z}^{d}$ and all $\omega\in\Omega$.

For simplicity of notation, we set $f_{t}\left(x,\omega\right):=P_{t}^{\omega}f\left(x,\omega\right)=E_{x}^{\omega}f\left(X_{t},\omega\right)$,
$t\geq0$, $x\in\mathbb{Z}^{d}$, $\omega\in\Omega$, where $E_{x}^{\omega}$
is the mean associated to $P_{x}^{\omega}$, and in many places we
shall omit the dependence in $\omega$ when there is no ambiguity.

Next, we define three different gradients. Denote by $e_{1},\dots,e_{d}$
the canonical orthonormal basis of $\mathbb{Z}^{d}$ (\textit{i.e.}\ for
all $i$, $e_{i}=(0,\dots,0,1,0,\dots,0)$ with $1$ at the $i^{th}$
coordinate), set $e_{-i}=-e_{i}$, $i=1,\dots,d$ and $e_{0}:=(0,\dots,0)$
for the origin. With that notation in hand, we define the following
local gradient of $g\colon\mathbb{Z}^{d}\rightarrow\mathbb{R}$: 
\[
D_{i}g(y):=g\left(y+e_{i}\right)-g\left(y\right)\qquad-d\leqslant i\leqslant d,\quad y\in\mathbb{Z}^{d}. \notag
\]
In particular, if $f\colon\mathbb{Z}^{d}\times\Omega\rightarrow\mathbb{R}$,
then $D_{i}f\left(y,\omega\right)=f\left(y+e_{i},\omega\right)-f\left(y,\omega\right)$,
$\omega\in\Omega$, and $D_{i}P_{x}^{\omega}\left(X_{t}=y\right)$
stands for the gradient applied to the mapping $y\mapsto P_{x}^{\omega}\left(X_{t}=y\right)$.
Similarly, 
\[
\nabla^{i}P_{x}^{m}\left(X_{t}=y\right):=P_{x+e_{i}}^{m}\left(X_{t}=y\right)-P_{x}^{m}\left(X_{t}=y\right)\qquad-d\leqslant i\leqslant d,\quad x,y\in\mathbb{Z}^{d}. \notag
\]
Therefore, the infinitesimal generator of a random walk with scheme
$m$ can be written as 
\[
\mathcal{L}^{m}f\left(x,\omega\right)=\sum_{-d\leqslant i\leqslant d}m_{x,x+e_{i}}D_{i}f\left(x,\omega\right)\qquad x\in\mathbb{Z}^{d},\quad\omega\in\Omega. \notag
\]

Finally, for $x\in\mathbb{Z}^{d}$, let $a\left(x\right)=\left\{ \omega_{x,x+e_{i}};\,1\leqslant i\leqslant d\right\} $,
$\overline{a}\left(x\right)=\{\omega_{e},e\in\mathbb{B}^{d}\}\setminus a(x)=\cup_{y\neq x}a\left(y\right)$
and $\mathbb{E}^{(x)}$ be the conditional expectation given $\overline{a}\left(x\right)$.
Then, for $y\in\mathbb{Z}^{d}$, we set 
\[
\partial_{x}f\left(y,\omega\right):=f\left(y,\omega\right)-\mathbb{E}^{(x)}\left[f\left(y,\omega\right)\right] \notag
\]
so that $\mathbb{E}^{(x)}[(\partial_{x}f)^{2}]$ is nothing else than
the variance of $f$ with respect to the conditional expectation $\mathbb{E}^{(x)}$.
Two sites $x,y$ are \emph{neighbors}, a property we denote by $x\sim y$,
if $\{x,y\}\in\mathbb{B}^{d}$. Also, given $f\colon\Omega\rightarrow\mathbb{R}$,
we say that $x$ is a neighbor of $y$ with respect to $f$, and write
$x\underset{f}{\sim}y$, if $a\left(x-y\right)\cap supp\left(f\right)\neq\emptyset$
(observe that this is not an equivalence relation). In particular,
observe that, if $x$ is not a neighbor of $y$ with respect to a
local translation invariant function $f$ (\textit{i.e.}\ if $a\left(x-y\right)\cap supp\left(f\right)=\emptyset$)
then $\partial_{x}f\left(y,\omega\right)=0$.

\section{Variance decay for unbounded conductances: proof of Theorem \ref{thm:MainTheorem}}

In this section we prove Theorem \ref{thm:MainTheorem}. The idea,
following \cite{gloria}, is to decompose the variance $\mathbb{E}\left(E_{x}^{\omega}\left[f\left(X_{t},\omega\right)\right]^{2}\right)$,
using Efron-Stein's inequality, into an infinite sum of terms of the
type $\mathbb{E}\left[\left(\partial_{y}P_{t}f\right)^{2}\right]$,
which, by Duhamel's formula, are split into two different terms that
need to be analyzed separately (the core of the proof). The point
in using Duhamel's Formula is to commute the operators $P_{t}$ and
$\partial_{y}$. The proof ends by applying some sort of Gronwall Lemma.

\begin{proof}{[}Proof of Theorem \ref{thm:MainTheorem}{]} Let $f\colon\Omega\to\mathbb{R}$
be a local translation invariant function with $\mathbb{E}[f]=0$,
and assume, by homogeneity and for simplicity, that $\mathbb{E}[f^{2}]=1$.
Following \cite{gloria}, we apply Efron-Stein's Inequality and the
Duhamel formula (that we recall below, in Lemma \ref{froid}, for
completeness) to bound $\mathbb{E}\left[f_{t}^{2}\right]$: 
\begin{align}
\mathbb{E}\left[f_{t}^{2}\right] & \leqslant\mathbb{E}\left[\sum_{y\in\mathbb{Z}^{d}}\left(\partial_{y}P_{t}f\right)^{2}\right]=\mathbb{E}\left[\sum_{y\in\mathbb{Z}^{d}}\left(P_{t}\partial_{y}f+\int_{0}^{t}P_{t-s}h_{s}\left(0,y,\omega\right)ds\right)^{2}\right]\notag\label{MainProof1}\\
 & \leqslant2\mathbb{E}\left[\sum_{y\in\mathbb{Z}^{d}}\left(P_{t}\partial_{y}f\right)^{2}\right]+2\mathbb{E}\left[\sum_{y\in\mathbb{Z}^{d}}\left(\int_{0}^{t}P_{t-s}h_{s}\left(0,y,\omega\right)ds\right)^{2}\right].
\end{align}
where $h_{s}\left(x,y,\omega\right):=\mathbb{E}^{(y)}\left[\mathcal{L}f_{s}\left(x,\omega\right)\right]-\mathcal{L}\mathbb{E}^{(y)}\left[f_{s}\left(x,\omega\right)\right]$,
$x,y\in\mathbb{Z}^{d}$, $\omega\in\Omega$, $s\geq0$. Next we analyze
each term of the right hand side of the latter separately and start
with the first one.

First recall that if $a\left(y-x\right)\cap supp\left(f\right)=\emptyset$
then $\partial_{y}f\left(x,\omega\right)=0$. Hence 
\begin{align}
\mathbb{E}\left[\sum_{y\in\mathbb{Z}^{d}}\left(P_{t}\partial_{y}f\right)^{2}\right] & =\mathbb{E}\left[\sum_{y}\left(\sum_{x\in\mathbb{Z}^{d}}P_{0}^{\omega}\left(X_{t}=x\right)\partial_{y}f\left(x,\omega\right)\right)^{2}\right]\notag\label{MainProof2}\\
 & \leqslant\#supp(f)\sum_{y}\sum_{x:y\underset{f}{\sim}x}\mathbb{E}\left[P_{0}^{\omega}\left(X_{t}=x\right)^{2}\partial_{y}f\left(x,\omega\right)^{2}\right].
\end{align}
By invariance by translation we have 
\begin{align*}
\mathbb{E}\left[P_{0}^{\omega}\left(X_{t}=x\right)^{2}\partial_{y}f\left(x,\omega\right)^{2}\right] & =\mathbb{E}\left[P_{0}^{\theta_{-x}\omega}\left(X_{t}=x\right)^{2}\partial_{y}f\left(x,\theta_{-x}\omega\right)^{2}\right]\\
 & =\mathbb{E}\left[P_{0}^{\omega}\left(X_{t}=-x\right)^{2}\partial_{y-x}f\left(0,\omega\right)^{2}\right].
\end{align*}
Therefore, changing variables (set $x'=y-x$ and $y'=x'-y$), it holds
\begin{align*}
\mathbb{E}\left[\sum_{y\in\mathbb{Z}^{d}}\left(P_{t}\partial_{y}f\right)^{2}\right] & \leqslant\#supp(f)\sum_{y}\sum_{x:y\underset{f}{\sim}x}\mathbb{E}\left[P_{0}^{\omega}\left(X_{t}=-x\right)^{2}\partial_{y-x}f\left(0,\omega\right)^{2}\right]\notag\\
 & \leqslant\#supp(f)\sum_{y}\sum_{x':x'\underset{f}{\sim}0}\mathbb{E}\left[P_{0}^{\omega}\left(X_{t}=x'-y\right)^{2}\partial_{x'}f\left(0,\omega\right)^{2}\right]\notag\\
 & \leqslant\#supp(f)\sum_{x':x'\underset{f}{\sim}0}\sum_{y'}\mathbb{E}\left[P_{0}^{\omega}\left(X_{t}=y'\right)^{2}\partial_{x'}f\left(0,\omega\right)^{2}\right].
\end{align*}
Finally, using Lemma $\ref{pro:DecroissanceProbaMarche}$ below and
the fact that $\mathbb{E}\left[\partial_{x'}f\left(0,\cdot\right)^{2}\right]\leqslant2\mathbb{E}\left[f\left(0,\cdot\right)^{2}\right]=2\mathbb{E}\left[f^{2}\right]=2$,
we conclude that, for some constant $C$ that depends only on $d$,
\begin{equation}
\mathbb{E}\left[\sum_{y\in\mathbb{Z}^{d}}\left(P_{t}\partial_{y}f\right)^{2}\right]\leqslant C\frac{\#supp\left(f\right)^{2}}{\left(t+1\right)^{d/2}}\label{term1}
\end{equation}

Next we focus on the second term in the right hand side of \eqref{MainProof1}.
Using Lemma \ref{pro:PropositionIntermediaire} we have 
\[
\mathbb{E}\left[\sum_{y\in\mathbb{Z}^{d}}\left(\int_{0}^{t}P_{t-s}h_{s}\left(0,y,\omega\right)ds\right)^{2}\right]=\mathbb{E}\left[\sum_{y\in\mathbb{Z}^{d}}\left(\int_{0}^{t}\sum_{i=1}^{d}D_{i}P_{0}^{\omega}\left(X_{t-s}=y\right)g_{s}\left(y,y,\omega,i\right)ds\right)^{2}\right]\notag
\]
so that, by Minkowski's integral inequality and the invariance by
translation, it holds 
\begin{align}
 & \mathbb{E}\left[\sum_{y\in\mathbb{Z}^{d}}\left(\int_{0}^{t}P_{t-s}h_{s}\left(0,y,\omega\right)ds\right)^{2}\right]^{\nicefrac{1}{2}}\notag\label{MainProof4}\\
 & \qquad\qquad\leqslant\sqrt{d}\int_{0}^{t}\left(\sum_{y}\sum_{i}\mathbb{E}\left[\left(D_{i}P_{0}^{\omega}\left(X_{t-s}=y\right)\right)^{2}g_{s}\left(y,y,\omega,i\right)^{2}\right]\right)^{\nicefrac{1}{2}}ds\notag\\
 & \qquad\qquad=\sqrt{d}\int_{0}^{t}\left(\sum_{y}\sum_{i}\mathbb{E}\left[\left(\nabla^{i}P_{0}^{\omega}\left(X_{t-s}=y\right)\right)^{2}g_{s}\left(0,0,\omega,i\right)^{2}\right]\right)^{\nicefrac{1}{2}}ds\notag\\
 & \qquad\qquad\leqslant\sqrt{2d}\int_{0}^{t}\left(\sum_{y}\sum_{i}\mathbb{E}\left[\left(\nabla^{i}P_{0}^{\omega}\left(X_{t-s}=y\right)\right)^{2}\mathbb{E}^{(0)}\left[\omega_{0,e_{i}}D_{i}f_{s}\left(0,\omega\right)\right]^{2}\right]\right.\notag\\
 & \qquad\qquad\quad+\left.\mathbb{E}\left[\omega_{0,e_{i}}^{2}\left(\nabla^{i}P_{0}^{\omega}\left(X_{t-s}=y\right)\right)^{2}\mathbb{E}^{(0)}\left[D_{i}f_{s}\left(0,\omega\right)\right]^{2}\right]\right)^{\nicefrac{1}{2}}ds
\end{align}
where $g_{s}\left(x,y,\omega,i\right):=\mathbb{E}^{(x)}\left[\omega_{y,y+e_{i}}D_{i}f_{s}\left(y,\omega\right)\right]-\omega_{y,y+e_{i}}\mathbb{E}^{(x)}\left[D_{i}f_{s}\left(y,\omega\right)\right]$,
$s\geq0$, $x,y\in\mathbb{Z}^{d}$, $\omega\in\Omega$ and $i=1,\dots,d$.
Therefore, using twice that $\left(a+b\right)^{2}\leqslant2a^{2}+2b^{2}$,
Lemma \ref{pro:DecroissanceProbaMarche} and Lemma \ref{lem:Lemmeclef}
guarantee that, for some constant $C$ that depends only on $d$,
it holds 
\begin{align}
 & \mathbb{E}\left[\sum_{y\in\mathbb{Z}^{d}}\left(\int_{0}^{t}P_{t-s}h_{s}\left(y,0,\omega\right)ds\right)^{2}\right]^{\nicefrac{1}{2}}\!\!\!\leqslant C\int_{0}^{t}\left(\left(t-s+1\right)^{-\frac{d}{2}}\sum_{i}\mathbb{E}\left[\mathbb{E}^{(0)}\left[\omega_{0,e_{i}}D_{i}f_{s}\left(0,\omega\right)\right]^{2}\right]\right.\notag\\
 & \qquad\qquad\quad+\left.\left(t-s+1\right)^{-d/2}\sum_{i}\mathbb{E}\left[\omega_{0,e_{i}}^{2}\mathbb{E}^{(0)}\left[D_{i}f_{s}\left(0,\omega\right)\right]^{2}\right]\right)^{1/2}ds\notag\\
 & \qquad\qquad\quad\leqslant\sqrt{2}C\int_{0}^{t}\left(t-s+1\right)^{-\nicefrac{d}{4}}\mathbb{E}\left[\omega_{0,e_{1}}^{2}\right]^{1/2}\left(-\partial_{s}\mathbb{E}\left[|f_{s}\left(0,\omega\right)|^{2}\right]\right)^{1/2}ds. \label{term2} \end{align}
Plugging \eqref{term1} and \eqref{term2} into \eqref{MainProof1},
we end up with 
\[
\mathbb{E}\left[f_{t}^{2}\right]^{\frac{1}{2}}\leqslant C'\#supp\left(f\right)\mathbb{E}\left[\omega_{0,e_{1}}^{2}\right]^{\frac{1}{2}}\left(\left(t+1\right)^{-\frac{d}{4}}+\int_{0}^{t}\left(t-s+1\right)^{-\frac{d}{4}}\left(-\partial_{s}\mathbb{E}\left[f_{s}^{2}\right]\right)^{\frac{1}{2}}ds\right)\notag
\]
for some constant $C'$ that depends only on $d$. The expected result
will finally follow from Lemma \ref{lem:Lemme15bis} with $a\left(t\right):=\mathbb{E}\left[f_{t}^{2}\right]^{\frac{1}{2}}$
and $\alpha=\frac{d}{4}$. Indeed, $a(t)$ is non-increasing since,
using classical computations for reversible Markov processes, $\partial_{t}\mathbb{E}\left[f_{t}^{2}\right]=2\mathbb{E}\left[f_{t}\mathcal{L}f_{t}\right]=-\sum_{i}\mathbb{E}\left[\omega_{0,e_{i}}(D_{i}f_{t})^{2}\right]\leq0$.
This ends the proof. \end{proof}

In the proof of Theorem \ref{thm:MainTheorem} we used the following
series of lemma. The first two lemmas are well known results from
Probability Theory and Analysis. The others are technical.

\begin{lem}{[}Efron-Stein's Inequality and the Duhamel formula{]}
\label{froid} The following holds. 
\begin{description}
\item [{(Efron-Stein's Inequality)}] Let $n>1$ and $f$ be a function
of $X_{1},...,X_{n}$, $n$ independent variables, then 
\[
\mathbb{V}ar\left(f\right)\leqslant\sum_{i=1}^{n}\mathbb{E}\left[\mathbb{V}ar^{(i)}\left(f\right)\right]\notag
\]
where $\mathbb{V}ar^{(i)}$ is the conditional variance given $\left\{ X_{1},...,X_{n}\right\} \backslash\left\{ X_{i}\right\} $. 
\item [{(Duhamel's Formula)}] For all $t\geqslant0$ and almost all $\omega\in\Omega$
it holds 
\[
\partial_{y}P_{t}f\left(x,\omega\right)=P_{t}\partial_{y}f\left(x,\omega\right)+\int_{0}^{t}P_{t-s}h_{s}\left(x,y,\omega\right)ds \notag
\]
where $h_{s}\left(x,y,\omega\right):=\mathbb{E}^{(y)}\left[\mathcal{L}f_{s}\left(x,\omega\right)\right]-\mathcal{L}\mathbb{E}^{(y)}\left[f_{s}\left(x,\omega\right)\right]$,
$x,y\in\mathbb{Z}^{d}$, $\omega\in\Omega$, $s\geq0$. 
\end{description}
\end{lem}

\begin{proof} Efron-Stein's Inequality, also called tensorisation
of the variance (see \textit{e.g.}\ \cite[Proposition 1.4.1]{ane}),
is a well known result following from Cauchy-Schwarz' inequality.
See \cite{latala} for an extension to general $\phi$-entropy, see
also \cite{Ma07}.

As for the Duhamel formula, we observe that 
\[
\partial_{y}P_{t}f\left(x,\omega\right)=P_{t}\partial_{y}f\left(x,\omega\right)+\int_{0}^{t}\partial_{s}P_{t-s}\partial_{y}P_{s}f\left(x,\omega\right)ds \notag
\]
which leads to the expected result, since $h_{s}\left(x,y,\omega\right)=\left(\partial_{y}\mathcal{L}-\mathcal{L}\partial_{y}\right)P_{s}f$
and $\partial_{t}P_{t}=\mathcal{L}P_{t}=P_{t}\mathcal{L}$. \end{proof}


\begin{lem} \label{pro:DecroissanceProbaMarche} There exists a constant
$C$ (that depends only on $d$) such that for 
all well-defined walk scheme $m$ and for all $t>0$ it holds 
\[
\sum_{y\in\mathbb{Z}^{d}}P_{0}^{m}\left(X_{t}=y\right)^{2}\leqslant C\left(t+1\right)^{-d/2}. \notag
\]
\end{lem}

\begin{proof} The proof of this inequality is a consequence of
the reversibility and the fact that the invariant measure is the uniform
measure. Indeed 
\[
\sum_{y\in\mathbb{Z}^{d}}P_{0}^{m}\left(X_{t}=y\right)^{2}=\sum_{y\in\mathbb{Z}^{d}}P_{0}^{m}\left(X_{t}=y\right)P_{y}^{m}\left(X_{t}=0\right)=P_{0}^{m}\left(X_{2t}=0\right) \notag
\]
which gives the desired result combining \cite[Theorem 2.1]{CKS}
and \cite[Proposition 10.2]{Mo99} in its first arXiv version. \end{proof}


\begin{lem} \label{pro:PropositionIntermediaire} Given $f\colon\mathbb{Z}^{d}\times\Omega\to\mathbb{R}$,
define $h_{s}\left(x,y,\omega\right):=\mathbb{E}^{(y)}\left[\mathcal{L}f_{s}\left(x,\omega\right)\right]-\mathcal{L}\mathbb{E}^{(y)}\left[f_{s}\left(x,\omega\right)\right]$
and $g_{s}\left(x,y,\omega,i\right)=\mathbb{E}^{(y)}\left[\omega_{x,x+e_{i}}D_{i}f_{s}\left(x,\omega\right)\right]-\omega_{x,x+e_{i}}\mathbb{E}^{(y)}\left[D_{i}f_{s}\left(x,\omega\right)\right]$,
$x,y\in\mathbb{Z}^{d}$, $\omega\in\Omega$, $s\geq0$ and $i=0,\dots,d$.
Then, for all $s,t>0$ and all $x,y$ it holds 
\[
P_{t}h_{s}\left(x,y,\omega\right)=-\sum_{i=1}^{d}D_{i}P_{x}\left(X_{t}=y\right)g_{s}\left(y,y,\omega,i\right). \notag 
\]
\end{lem}

\begin{proof} Recall the definition of $D_{i}$. On the one hand,
by definition, we have 
\[
\mathbb{E}^{(y)}\left[\mathcal{L}f_{s}\left(x,\omega\right)\right]=\sum_{i=1}^{d}\mathbb{E}^{(y)}\left[\omega_{x,x+e_{i}}\left(D_{i}f_{s}\left(y,\omega\right)\right)\right]-\mathbb{E}^{(y)}\left[\omega_{x,x-e_{i}}\left(D_{i}f_{s}\left(x-e_{i},\omega\right)\right)\right] \notag
\]
and 
\[
\mathcal{L}\mathbb{E}^{(y)}\left[f_{s}\left(x,\omega\right)\right]=\sum_{i=1}^{d}\omega_{x,x+e_{i}}\mathbb{E}^{(y)}\left[\left(D_{i}f_{s}\left(x,\omega\right)\right)\right]-\omega_{x,x-e_{i}}\mathbb{E}^{(y)}\left[\left(D_{i}f_{s}\left(x-e_{i},\omega\right)\right)\right]. \notag
\]
On the other hand, $\omega_{x,x+e_{i}}$ is $\bar{a}\left(y\right)$-measurable
iff $x\neq y$ and $\omega_{x,x-e_{i}}$ is $\bar{a}\left(y\right)$-measurable
iff $x-e_{i}\neq y$. Therefore, 
\[
h_{s}\left(x,y,\omega\right)=\begin{cases}
\sum_{i=1}^{d}g_{s}\left(y,y,\omega,i\right) & \mbox{if }x=y\\
-g_{s}\left(y,y,\omega,i\right) & \mbox{if }x-e_{i}=y,\; i=1,\dots,d\\
0 & \mbox{otherwise}. \notag
\end{cases}
\]
It finally follows that 
\begin{eqnarray*}
P_{t}h_{s}\left(x,y,\omega\right) & = & \sum_{z\in\mathbb{Z}^{d}}P_{x}\left(X_{t}=z\right)h_{s}\left(z,y,\omega\right)\\
 & = & P_{x}\left(X_{t}=y\right)h_{s}\left(y,y,\omega\right)-\sum_{i=1}^{d}P_{x}\left(X_{t}=y+e_{i}\right)h_{s}\left(y+e_{i},y,\omega\right)\\
 & = & \sum_{i=1}^{d}\left(P_{x}\left(X_{t}=y\right)-P_{x}\left(X_{t}=y+e_{i}\right)\right)g_{s}\left(y,y,\omega,i\right)
\end{eqnarray*}
which is the desired result. \end{proof}


\begin{lem} \label{lem:Lemmeclef} Let $f\colon\mathbb{Z}^{d}\times\Omega\to\mathbb{R}$.
Then, for all $s>0$ it holds 
\begin{align*}
\sum_{i=1}^{d}\mathbb{E}\left[\left(\mathbb{E}^{(0)}\left[\omega_{0,e_{i}}D_{i}f_{s}\left(0,\omega\right)\right]\right)^{2}\right] & \leqslant-\mathbb{E}\left[\omega_{0,e_{1}}^{2}\right]\partial_{s}\mathbb{E}\left[f_{s}\left(0,\omega\right)^{2}\right],\\
\sum_{i=1}^{d}\mathbb{E}\left[\left(\omega_{0,e_{i}}\mathbb{E}^{(0)}\left[D_{i}f_{s}\left(0,\omega\right)\right]\right)^{2}\right] & \leqslant-\mathbb{E}\left[\omega_{0,e_{1}}^{2}\right]\partial_{s}\mathbb{E}\left[f_{s}\left(0,\omega\right)^{2}\right].
\end{align*}
\end{lem}

\begin{rem} The proof actually leads to a better bound in the first
inequality above. Indeed, one can replace the second moment of the
conductance by the first moment. 
This refinement will anyway not be useful for our purpose. \end{rem}

\begin{proof}{[}Proof of Lemma \ref{lem:Lemmeclef}{]} Using Cauchy-Schwarz'
inequality and the fact that $\omega_{0,e_{i}}$ is $a\left(0\right)$-measurable,
we have 
\begin{align*}
\mathbb{E}\left[\left(\mathbb{E}^{(0)}\left[\omega_{0,e_{i}}D_{i}f_{s}\left(0,\omega\right)\right]\right)^{2}\right] & \leqslant\mathbb{E}\left[\mathbb{E}^{(0)}\left[\omega_{0,e_{i}}\right]\mathbb{E}^{(0)}\left[\omega_{0,e_{i}}\left(D_{i}f_{s}\left(0,\omega\right)\right)^{2}\right]\right]\\
 & =\mathbb{E}\left[\omega_{0,e_{i}}\right]\mathbb{E}\left[\omega_{0,e_{i}}\left(D_{i}f_{s}\left(0,\omega\right)\right)^{2}\right].
\end{align*}
Using that $\partial_{t}\mathbb{E}\left[f_{t}^{2}\right]=\partial_{t}\mathbb{E}\left[f_{t}(0,\omega)^{2}\right]=-\sum_{i=-d}^{d}\mathbb{E}\left[\omega_{0,e_{i}}\left(D_{i}f_{s}\left(0,\omega\right)\right)^{2}\right]$
(a classical consequence of the reversibility) and summing over $i=1,\dots,d$,
we get 
\[
\sum_{i=1}^{d}\mathbb{E}\left[\left(\mathbb{E}^{(0)}\left[\omega_{0,e_{i}}D_{i}f_{s}\left(0,\omega\right)\right]\right)^{2}\right]\leqslant-\mathbb{E}\left[\omega_{0,e_{1}}\right]\partial_{s}\mathbb{E}\left[f_{s}\left(0,\omega\right)^{2}\right]\notag 
\]
which leads to the first inequality since $\omega_{0,e_{1}}\geq1$.

For the second inequality by conditioning and Jensen's inequality, we have 
\begin{align*}
\mathbb{E}\left[\left(\omega_{0,e_{i}}\mathbb{E}^{(0)}\left[D_{i}f_{s}\left(0,\omega\right)\right]\right)^{2}\right] & =\mathbb{E}\left[\mathbb{E}^{(0)}\left[\omega_{0,e_{i}}^{2}\right]\left(\mathbb{E}^{(0)}\left[D_{i}f_{s}\left(0,\omega\right)\right]\right)^{2}\right]\\
 & =\mathbb{E}\left[\omega_{0,e_{i}}^{2}\right]\mathbb{E}\left[D_{i}f_{s}\left(0,\omega\right)^{2}\right].
\end{align*}
Now, since $\omega_{0,e_{1}}\geq1$, one has $\sum_{i=1}^{d}\mathbb{E}\left[D_{i}f_{s}\left(0,\omega\right)^{2}\right]\leq\sum_{i=1}^{d}\mathbb{E}\left[\omega_{0,e_{1}}D_{i}f_{s}\left(0,\omega\right)^{2}\right]=-\partial_{s}\mathbb{E}\left[f_{s}(0,\omega)^{2}\right]$
which leads to the desired result and ends the proof of the lemma.
\end{proof}

The last lemma is related to Lemma 15 of \cite{gloria}. However,
due to our specific setting, considering $p=1/2$, its proof is a
bit simpler. We give it for completeness.


\begin{lem}{[}\cite{gloria}{]} \label{lem:Lemme15bis} Let $\alpha>1/2$
and $a\colon\mathbb{R}^{+}\to\mathbb{R}^{+}$ be a $\mathcal{C}^{1}$
non-increasing function. Let $b(t)=\sqrt{-2a'(t)a(t)}$, $t\geq0$.
Assume that 
\begin{align*}
a\left(t\right) & \leqslant C\left(\left(t+1\right)^{-\alpha}+\int_{0}^{t}\left(t-s+1\right)^{-\alpha}b\left(s\right)ds\right)\qquad\forall t\geq0
\end{align*}
for some constant $C$. Then there exists a constant $C_{\alpha}$
that depends only on $\alpha$ such that $a\left(t\right)\leqslant C_{\alpha}\max(C,a(0))\left(t+1\right)^{-\alpha}$.
\end{lem}

\begin{proof} Throughout the proof we use that $u\lesssim v$ if
there exists a constant $A$ that depends only on $\alpha$ such that
$u\leq Av$. The expected result will follow from the fact that, for
some $t_{o}>0$ that will be chosen later on, $[t_{o},\infty)\ni t\mapsto\Lambda\left(t\right):=\sup_{t_{o}\leq s\leqslant t}\left(s+1\right)^{\alpha}a\left(s\right)$
is a bounded function, bounded by $C'\max(C,a(0))$ for some $C'$
depending only on $\alpha$. Indeed, since $a$ is non-increasing
$a(s)\leq a(0)$ for any $s\in[0,t_{o}]$ which, together with the
bound on $\Lambda$ would lead to the desired conclusion.

Our starting point is the following inequality obtained using that
$a$ is non-increasing. 
\begin{align}
a\left(t\right) & \leqslant\frac{2}{t}\int_{\frac{t}{2}}^{t}a\left(u\right)du\leqslant\frac{2C}{t}\int_{\frac{t}{2}}^{t}\frac{du}{\left(u+1\right)^{\alpha}}+\frac{2}{t}\int_{\frac{t}{2}}^{t}\int_{0}^{u}\frac{b(s)dsdu}{\left(u+1-s\right)^{\alpha}}\leq\frac{C}{(\frac{t}{2}+1)^{\alpha}}\nonumber \\
 & +\frac{2}{t}\int_{\frac{t}{2}}^{t}\int_{0}^{T}\frac{b(s)dsdu}{(u+1-s)^{\alpha}}+\frac{2}{t}\int_{\frac{t}{2}}^{t}\int_{T}^{u/2}\frac{b(s)dsdu}{(u+1-s)^{\alpha}}+\frac{2}{t}\int_{\frac{t}{2}}^{t}\int_{u/2}^{u}\frac{b(s)dsdu}{(u+1-s)^{\alpha}}\nonumber \\
 & =\frac{C}{(\frac{t}{2}+1)^{\alpha}}+I+II+III \label{start26}
\end{align}
where $T\in[1,t/4]$ is a parameter that will be chosen later on.
In order to bound $I,II$ and $III$, we will repeatedly use the following
bounds, that holds for all $T_{1}<T_{2}$ and whose proof is given
below 
\begin{equation}
\int_{T_{1}}^{T_{2}}b\left(s\right)ds\lesssim\begin{cases}
\sqrt{T_{2}-T_{1}}a\left(T_{1}\right)\\
\Lambda\left(T_{2}\right)T_{1}^{\frac{1}{2}-\alpha} & \mbox{ if }T_{1}>0.
\end{cases}\label{estimebs}
\end{equation}
To prove the first inequality, we use Cauchy-Schwarz' inequality.
Namely 
\begin{align*}
\frac{1}{T_{2}-T_{1}}\left(\int_{T_{1}}^{T_{2}}b\left(s\right)ds\right)^{2} & \leqslant\int_{T_{1}}^{T_{2}}b^{2}\left(s\right)ds=-\int_{T_{1}}^{T_{2}}\frac{d}{ds}a^{2}\left(s\right)ds\leqslant a\left(T_{1}\right)^{2}.
\end{align*}
To prove the second inequality, we repeatedly use the latter. set
$N=\lceil\log_{2}\left(\nicefrac{T_{2}}{T_{1}}\right)\rceil$, then we have 
\begin{align*}
\int_{T_{1}}^{T_{2}}b\left(s\right)ds & =\sum_{n=0}^{N-1}\int_{2^{n}T_{1}}^{2^{n+1}T_{1}}b\left(s\right)ds\leqslant\sum_{n=0}^{N-1}\sqrt{2^{n}T_{1}}a\left(2^{n}T_{1}\right)\leqslant\sum_{n=0}^{N-1}\left(2^{n}T_{1}\right)^{\frac{1}{2}-\alpha}\Lambda\left(2^{n}T_{1}\right)\\
 & \leqslant\Lambda\left(T_{2}\right)T_{1}^{\frac{1}{2}-\alpha}\sum_{n=0}^{N-1}\left(2^{\frac{1}{2}-\alpha}\right)^{n}\lesssim\Lambda\left(T_{2}\right)T_{1}^{\frac{1}{2}-\alpha}.
\end{align*}

Now, since $T\leq t/4$ and thanks to \eqref{estimebs}, we have 
\[
I:=\frac{2}{t}\int_{\frac{t}{2}}^{t}\int_{0}^{T}\frac{b(s)dsdu}{(u+1-s)^{\alpha}}\leqslant\frac{2}{t}\int_{\frac{t}{2}}^{t}\int_{0}^{T}\frac{b(s)dsdu}{(\frac{t}{4}+1)^{\alpha}}=\frac{\int_{0}^{T}b(s)ds}{(\frac{t}{4}+1)^{\alpha}}\lesssim\frac{\sqrt{T}a(0)}{(t+1)^{\alpha}}.
\notag \]
Again using \eqref{estimebs} we have 
\[
II:=\frac{2}{t}\int_{\frac{t}{2}}^{t}\int_{T}^{u/2}\frac{b(s)dsdu}{(u+1-s)^{\alpha}}\leq\frac{2}{t}\int_{\frac{t}{2}}^{t}\frac{\int_{T}^{u/2}b(s)ds}{(\frac{t}{4}+1)^{\alpha}}du\leq\frac{\frac{2}{t}\int_{\frac{t}{2}}^{t}\Lambda(\frac{u}{2})T^{\frac{1}{2}-\alpha}du}{(\frac{t}{4}+1)^{\alpha}}\lesssim\frac{\Lambda(t)T^{\frac{1}{2}-\alpha}}{(t+1)^{\alpha}}.
\notag \]
In order to bound the third term,
which is more intricate, we first use the Fubini Theorem, noting that $\mathds{1}_{ t/2 \leqslant u \leqslant t } \mathds{1}_{ u/2 \leqslant s \leqslant u} \leqslant \mathds{1}_{t/4\leqslant s \leqslant t} \mathds{1}_{s\leqslant u \leqslant t}$, 
and that $\alpha > 1/2$ to get 
\begin{eqnarray}
III & \leqslant & \frac{2}{t}\int_{\frac{t}{4}}^{t}\int_{s}^{t}\left(u+1-s\right)^{-\alpha}dub\left(s\right)ds =  \frac{2}{t}\int_{\frac{t}{4}}^{t}\int_{0}^{t-s}\left(t'+1\right)^{-\alpha}dt'b\left(s\right)ds\notag\\
 & \leqslant & \frac{2}{t}\int_{0}^{t}\left(t'+1\right)^{-\alpha}dt'\times\int_{\frac{t}{4}}^{t}b\left(s\right)ds \lesssim (t+1)^{-\frac{1}{2}-\alpha}\Lambda\left(t\right)\notag
\end{eqnarray}
Therefore, plugging the previous bounds on $I,II$
and $III$ into \ref{start26}, it holds 
\[
(1+t)^{\alpha}a(t)\leq A\sqrt{T}\max(C,a(0))+A\Lambda(t)\left(T^{\frac{1}{2}-\alpha}+\frac{1}{\sqrt{t}}\right)
\notag \]
for some constant $A$ that depends only on $\alpha$. Now, since
$\alpha>1/2$, there exist $t_{o}\geq4$ and $T\geq1$ such that for
all $t\geq t_{o}$, $1/\sqrt{t}\leq1/(4A)$ and $T^{\frac{1}{2}-\alpha}\leq1/(4A)$
so that, taking the supremum and using the monotonicity of $\Lambda$,
it holds $\Lambda(t)\leq A\sqrt{T}\max(C,a(0))+\frac{1}{2}\Lambda(t)$
for all $t\geq t_{o}$ which leads to the desired conclusion. The
proof of the lemma is complete. \end{proof}

\section{Additional remarks}

\subsection{Completely monotonic functions}

In this section, we prove some results on $P_{x}^{m}\left(X_{t}=y\right)$,
for a given walk scheme $m$, using the notion of completely monotonic
functions. Recall that a function $f\colon(0,\infty)\to\mathbb{R}$
is said to be completely monotonic if it possesses derivatives $f^{(n)}$
of all orders and if $(-1)^{n}f^{(n)}(x)\geq0$ for all $x>0$ and
all $n=0,1,2,\dots$ (see \textit{e.g.} \cite{feller,miller}).

\begin{prop} \label{cor:CorollaireDeriveeCM} Let $C,\alpha>0$.
Assume that $f\colon(0,\infty)\to\mathbb{R}$ is a completely monotonic
function satisfying for all $t>0$, $f(t)\leqslant\frac{C}{t^{\alpha}}$.
Then, for all $t>0$, it holds $-f'(t)\leq\frac{C'}{t^{\alpha+1}}$
for some constant $C'$ that depends only on $C$ and $\alpha$. \end{prop}

\begin{rem} At the price of some technicalities, the above result
can be extended to more general decay (\textit{i.e.} replacing $1/t^{\alpha}$
by some general completely monotonic function $g$ with $\lim_{\infty}g=0$).
\end{rem}

\begin{proof} Without loss of generality assume that $f\left(0\right)=1$.
It is well known (see \cite{feller,widder}) that $f$ is the Laplace
transform of a positive random variable $X$, namely that, for all
$t>0$, $f\left(t\right)=\mathbb{E}\left[e^{-tX}\right]$ where $\mathbb{E}$
denotes the mean of the law of $X$. Using the Markov Inequality we
get for all $\lambda>0$ and all $x>0$ 
\[
\mathbb{P}\left(X\leqslant x\right)\leqslant\mathbb{E}\left[e^{-\lambda X}\right]e^{\lambda x}\leqslant C\exp\left(\lambda x-\alpha\log\lambda\right).
\notag \]
Optimizing over the $\lambda>0$ we get (since $\inf_{\lambda>0}\{\lambda x-\alpha\log\lambda\}=\alpha-\alpha\log\alpha+\alpha\log x$
(the minimum being reached at $\lambda=\alpha/x$)) $\mathbb{P}\left(X\leqslant x\right)\leqslant Ce^{\alpha}\left(\frac{x}{\alpha}\right)^{\alpha}$.
Therefore, using Fubini's theorem, for all $t>0$ it holds 
\begin{align*}
-f'(t) & =\mathbb{E}\left[Xe^{-tX}\right]=\mathbb{E}\left[X\int_{X}^{\infty}te^{-tx}dx\right]=\int_{0}^{\infty}te^{-tx}\mathbb{E}\left[X\mathds{1}_{X\leq x}\right]dx\\
 & \leq\int_{0}^{\infty}txe^{-tx}\mathbb{P}(X\leq x)dx\leq\frac{C_{\alpha}}{t^{\alpha}}\int_{0}^{\infty}(tx)^{\alpha+1}e^{-tx}dx=\frac{C_{\alpha}\Gamma(\alpha+2)}{t^{\alpha+1}}
\end{align*}
which ends the proof. \end{proof}

\begin{cor} \label{pro:DecroissanceProbaMarche2} There exists a
constant $C$ such that for all well-defined walk scheme $m$
and all $t>0$ it holds $\sum_{y\in\mathbb{Z}^{d},\left|z\right|=1}\left(P_{0}^{m}\left(X_{t}=y\right)-P_{0}^{m}\left(X_{t}=y+z\right)\right)^{2}\leqslant C/\left(t+1\right)^{\frac{d}{2}+1}$.
\end{cor}

\begin{proof} For $p\geq1$, set $\Vert f\Vert_{p}^{p}:=\sum_{x}\vert f\left(x\right)\vert ^{p}$
and, given an operator $P$ acting on functions, $\Vert P\Vert_{p\rightarrow q}=\sup\Vert Pf\Vert_{q}$,
where the supremum is taken over all $f$ with $\Vert f\Vert_{p}=1$.

First, we observe that the quantity is bounded as a consequence of Lemma \ref{pro:DecroissanceProbaMarche}.
As a matter of fact, for all $y\in\mathbb{Z}^{d}$ and $t\geqslant0$,
$\sum_{x}{P_{y}^{m}}\left(X_{t}=x\right)^{2}=P_{y}^{m}\left(X_{t}=y\right)\leqslant1$.

Then, note that $t\mapsto\Vert P_{t}^{m}f\Vert_{2}^{2}$ is a completely monotonic function.
It is due to the fact that $\mathcal{L}^{m}$ is self-adjoint ($<f,\mathcal{L}^{m}g> = <\mathcal{L}^{m}f,g>$)
and defines a negative quadratic form ($<\mathcal{L}^{m}f,f>\leqslant0$).
Indeed, for any function $f$, 
$\partial_{t}<P_{t}^{m}f,P_{t}^{m}f>=2<\mathcal{L}P_{t}^{m}f,P_{t}^{m}f>\leqslant0$,
$\partial_{t}\partial_{t}<P_{t}^{m}f,P_{t}^{m}f>=4<\mathcal{L}^{m}P_{t}^{m}f,\mathcal{L}^{m}P_{t}^{m}f>\geqslant0$,
and because $\mathcal{L}^{m}$ commutes with $P_{t}^{m}$, 
we can conclude repeating these last arguments with the function $\mathcal{L}^{m}f$.
Hence, using Proposition \ref{cor:CorollaireDeriveeCM},
there exists a constant $C'$ that depends only on $d$ and $C$ such
that for all $f\in\ell^{1}$, $\frac{d}{dt}\Vert P_{t}^{m}f\Vert_{2}^{2}\leqslant C'/t^{\frac{d}{2}+1}$.

On the other hand, by definition of $\mathcal{L}^{m}$ and $P_{t}^{m}$,
we have 
\begin{equation}
\frac{d}{dt}\Vert P_{t}^{m}f\Vert_{2}^{2}=-\sum_{x,|z|=1}m_{x,x+z}\left(P_{t}^{m}f\left(x+z\right)-P_{t}^{m}f\left(x\right)\right)^{2}\label{eq:fm2}
\end{equation}
Observe that for $f=\mathds{1}_{\{0\}}$ (the function that equals
1 at $0$ and 0 elsewhere), we have $f\in\ell^{1}$ 
and $P_{t}^{m}f(x)=\sum_{y\in\mathbb{Z}^{d}}P_{x}^{m}(X_{t}=y)f(y)=P_{0}^{m}(X_{t}=x)$
which plugged in $\eqref{eq:fm2}$ gives the desired result. \end{proof}

\subsection{Gloria, Neukamm and Otto with a fixed walk scheme}

Using the monotonic function approach above, we can prove a stronger
decay of the variance of the environment view by the particle, when
the function $f$ is the divergence of an other function, but only
when the walk scheme $m$ is fixed. This is a result (much weaker
but) in the spirit of \cite{gloria}.

\begin{prop} There exists a constant $C>0$ such that for almost all
walk scheme $m$, all $t\geqslant0$ and all function $f=D_{i}g$,
where $-d\leqslant i\leqslant d$, $g$ is local, translation-invariant
and $\mathbb{E}[g^{2}]<\infty$, it holds $\mathbb{E}\left[\left(P_{t}^{m}f\left(0,\omega\right)\right)^{2}\right]\leqslant C\#supp(g)^{2}\frac{\mathbb{E}[g^{2}]}{(t+1)^{\frac{d}{2}+1}}$.
\end{prop}

\begin{proof} Using the Cauchy-Schwarz Inequality and that $\mathbb{E}\left[g\left(x,\omega\right)g\left(y,\omega\right)\right]=0$
as soon as $supp\left(g(x,\omega)\right)\cap supp\left(g(y,\omega)\right)=\emptyset$
we have 
\begin{align*}
 & \mathbb{E}\left[\left(P_{t}^{m}f\left(0,\omega\right)\right)^{2}\right]=\mathbb{E}\left[\left(\sum_{x\in\mathbb{Z}^{d}}\left(P_{0}^{m}\left(X_{t}=x-e_{i}\right)-P_{0}^{m}\left(X_{t}=x\right)\right)g\left(x,\omega\right)\right)^{2}\right]\\
 & \leqslant\#supp\left(g\right)^{2}\sum_{x\in\mathbb{Z}^{d}}\left(P_{0}^{m}\left(X_{t}=x\right)-P_{0}^{m}\left(X_{t}=x+e_{i}\right)\right)^{2}\mathbb{E}\left[g^{2}\right]\notag
\end{align*}
which, combined with Corollary \ref{pro:DecroissanceProbaMarche2}
gives the expected result. \end{proof}

\subsection{Discussion about the polynomial decay}
In the introduction, we told that $t^{-d/2}$ is the optimal decay for 
local functions in $L^2$. Indeed, in \cite{DFGW}, they proved that, under the annealed law, 
the walker $X_{t}$ converges to a Brownian motion, suggesting a diffusive behaviour and this
rate of decay as optimal. In \cite{Mo99}, it is suggested that, using spectral theory, we can construct a non-local function $f$ such that $\mathbb{E}[f^{2}_{t}]$ decays as fast or as slow as we want. We note here that one can construct a local function with
a faster decay than $t^{-d/2}$. Indeed, because $\mathbb{E}[f^{2}_{t}]$ is a completely monotonic function, 
it is a consequence of corrolary \ref{cor:CorollaireDeriveeCM}.
For example, consider a function $g\in L^2$ such that $f=\mathcal{L}g$, then
$$\mathbb{E}[f^{2}_{t}]=\mathbb{E}[(\mathcal{L}g_{t})^{2}]=\partial_{t}\partial_{t}\mathbb{E}[g_{t}^{2}]\leqslant C_{g} t^{-d/2-2} $$
Iterating the process, considering $f=\mathcal{L}^{n}g$, we have that $\mathbb{E}[f^{2}_{t}]\leqslant C_{g}t^{-d/2-2n}$.
Note that in this case, $f$ might not be in $L^{2}$, even if $f_{t}$ would.

\subsection*{Acknowledgement} We thank Cyril Roberto and Julien Bureaux.

\bibliographystyle{alpha}
\bibliography{bibtex}

\end{document}